\documentclass[12pt]{amsart}
\usepackage{amssymb, amstext, amscd, amsmath}

\makeatletter
\def\@cite#1#2{{\m@th\upshape\bfseries%
[{#1\if@tempswa{\m@th\upshape\mdseries, #2}\fi}]}}
\makeatother

\theoremstyle{plain}
\newtheorem{thm}{Theorem}[section]
\newtheorem{cor}[thm]{Corollary}
\newtheorem{prop}[thm]{Proposition}
\newtheorem{lem}[thm]{Lemma}

\theoremstyle{definition}
\newtheorem{rem}[thm]{Remark}

\newtheorem{defn}[thm]{Definition}

\newtheorem{eg}[thm]{Example}

\newcommand{\bB}{{\mathbb{B}}}
\newcommand{\bC}{{\mathbb{C}}}
\newcommand{\bD}{{\mathbb{D}}}

\newcommand{\bT}{{\mathbb{T}}}
\newcommand{\bZ}{{\mathbb{Z}}}

  \newcommand{\A}{{\mathcal{A}}}
  \newcommand{\B}{{\mathcal{B}}}

  \newcommand{\E}{{\mathcal{E}}}

\renewcommand{\H}{{\mathcal{H}}}

  \newcommand{\K}{{\mathcal{K}}}

  \newcommand{\M}{{\mathcal{M}}}
  
\renewcommand{\O}{{\mathcal{O}}}

  \newcommand{\U}{{\mathcal{U}}}

\renewcommand{\phi}{\varphi}
\newcommand{\upchi}{{\raise.35ex\hbox{\ensuremath{\chi}}}}

\newcommand{\fA}{{\mathfrak{A}}}
\newcommand{\fB}{{\mathfrak{B}}}

\newcommand{\fJ}{{\mathfrak{J}}}
\newcommand{\fK}{{\mathfrak{K}}}
\newcommand{\fL}{{\mathfrak{L}}}

\newcommand{\fR}{{\mathfrak{R}}}

\newcommand{\fs}{{\mathfrak{s}}}

\newcommand{\fu}{{\mathfrak{u}}}

\newcommand{\bx}{{\mathbf{x}}}

\newcommand{\rC}{{\mathrm{C}}}

\newcommand{\qand}{\quad\text{and}\quad}
\newcommand{\qif}{\quad\text{if}\quad}
\newcommand{\qfor}{\quad\text{for}\quad}
\newcommand{\qforal}{\quad\text{for all}\quad}

\newcommand{\AND}{\text{ and }}
\newcommand{\FOR}{\text{ for }}

\newcommand{\ad}{\operatorname{ad}}

\newcommand{\Aut}{\operatorname{Aut}}

\newcommand{\id}{{\operatorname{id}}}

\newcommand{\ran}{\operatorname{Ran}}
\newcommand{\spn}{\operatorname{span}}

\DeclareMathOperator*{\sotlim}{\textsc{sot}--lim}

\newcommand{\ca}{\mathrm{C}^*}
\newcommand{\cenv}{\mathrm{C}^*_{\text{env}}}

\newcommand{\Fn}{\mathbb{F}_n^+}
\newcommand{\Fock}{{\ell^2(\Fn)}}
\newcommand{\lip}{\langle}
\newcommand{\rip}{\rangle}
\newcommand{\ip}[1]{\langle #1 \rangle}
\newcommand{\mt}{\varnothing}
\newcommand{\ol}{\overline}

\newcommand{\sot}{\textsc{sot}}
\newcommand{\wot}{\textsc{wot}}
\newcommand{\ltwo}{\ell^2}

\newcommand{\AD}{A(\bD)}

\begin{document}

\title[Dilations]{Dilating covariant representations of the non-commutative disc algebras}

\author[K.R.Davidson]{Kenneth R. Davidson}
\address{Pure Math.\ Dept.\\U. Waterloo\\Waterloo, ON\;
N2L--3G1\\CANADA}
\email{krdavids@uwaterloo.ca}

\author[E.G.Katsoulis]{Elias~G.~Katsoulis}
\address{Dept. Math.\\ East Carolina University\\
Greenville, NC 27858\\USA}
\email{KatsoulisE@mail.ecu.edu}

\begin{abstract}
Let $\phi$ be an isometric automorphism of the non-comm\-ut\-ative
disc algebra $\fA_n$ for $n \geq 2$.
We show that every contractive covariant representation of $(\fA_n, \phi)$ dilates to a
unitary covariant representation of $(\O_n, \phi)$.
Hence the C*-envelope of the semicrossed product
$\fA_n \times_{\phi} \bZ^+$ is  $\O_n \times_{\phi} \bZ$.
\end{abstract}

\thanks{First author partially supported by an NSERC grant.}
\thanks{Second author was partially supported by a grant from ECU}

\date{}
\maketitle

\section{Introduction}\label{S:intro}

In this paper, we continue our study of the crossed product $\fA_n \times_{\phi} \bZ^+$
of the non-commutative disk algebra $\fA_n$ by an isometric automorphism $\phi$.
These semicrossed products were introduced in \cite{DKn}
as universal algebras for the contractive covariant representations of $(\fA_n, \phi)$,
where we showed there that the isomorphic class of $\fA_n \times_{\phi} \bZ^+$
is determined by the analytic conjugacy class of $\phi$.
Recall that the isometric automorphisms of $\fA_n$ come from
the natural action of the group $\Aut(\bB_n)$ of conformal automorphisms of the
unit ball $\bB_n$ of $\bC^n$ on the character space $\ol{\bB}_n$ of $\fA_n$.

The universality of $\fA_n \times_{\phi} \bZ^+$ allows for a rich representation theory and
this was a key component for classifying these algebras.
On the other hand, it is the universality of $\fA_n \times_{\phi} \bZ^+$ that raises the
problem of finding concrete faithful representations.
This is the main theme of this paper.
As we shall see, the C*-envelope of $\fA_n \times_{\phi} \bZ^+$ is
$\O_n \times_{\phi} \bZ$, where $\O_n$ denotes the Cuntz algebra, with generators
going to generators.
Using the theory of gauge invariant uniqueness for C*-crossed product C*-algebras,
we obtain a concrete faithful representation for $\fA_n \times_{\phi} \bZ^+$.

The proof of this fact relies on a dilation theorem.  We show first that every
completely contractive representation of $\fA_n \times_{\phi} \bZ^+$ dilates
to a unitary system in which the image of the row of generators
$[L_1\ \dots\ L_n]$ of $\fA_n$ is a row isometry
and the intertwining operator implementing the automorphism is unitary.
This is not sufficient for our purposes, because the C*-algebra generated by a
row isometry is either $\O_n$ or the Cuntz-Toeplitz algebra $\E_n$.
We need to further dilate the latter representations to unitary dilations
of Cuntz type.  We thereby show that these are the only maximal representations
of the semicrossed product.  So using the Dritschel--McCullough approach \cite{DMc}
to the C*-envelope, we are able to obtain the desired conclusion.

Using a result of Kishimoto \cite{Kish}, we show that in the case
where $\phi$ is aperiodic, the C*-algebra
$\O_n \times_{\phi} \bZ$ is simple.
In this case, the quotient $\ca(\E_n, U_\phi)/\fK$, where $U_\phi$ is the Voiculescu
unitary implementing $\phi$ on $\E_n$, is $*$-isomorphic to $\O_n \times_{\phi} \bZ$.

There is an extensive body of work studying dynamical systems via an associated
operator algebras going back to work of von Neumann.  The use of nonself-adjoint
operator algebras in this area begins with seminal work of Arveson \cite{Arv} and
Arveson--Josephson \cite{ArvJ}.  This was put into the abstract setting of
semi-crossed products by Peters \cite{Pet}.
See \cite{DKsurvey} for an overview of some of the recent work in this area.

There is also a large literature on dilation theory for various
nonself-adjoint operator algebras going back to seminal work of Sz.Nagy.
Arveson \cite{Arv1} established dilation theory as an essential tool for studying
nonself-adjoint operator algebras.
Work on abstract semicrossed products began with work of Peters \cite{Pet}.
He concentrated on the action of an endomorphism on a C*-algebra,
and here the theory works well.  Specifically one can define a
family of natural orbit representations and show that these produce a
faithful (completely isometric) representation of the semicrossed product.
This can be used to find explicit information about the C*-envelope.
See Peters \cite{Pet3} for the one variable case and \cite{DavR} for
the C*-envelope a multivariable dynamical system.

Muhly and Solel developed an extensive theory of certain
nonself-adjoint operator algebras called tensor algebras of
a C*-correspondences \cite{MS1,MS2,MS3}.
They showed,  under certain hypotheses, that the C*-envelope of the
tensor algebra is the Cuntz--Pimsner C*-algebra built from the correspondence.
This result was extended by Fowler, Muhly and Raeburn \cite{FMR}
to the case when the left action is faithful and strict.  Finally
the second author and Kribs \cite{KK} removed those restrictions.

The semicrossed product of $\fA_n$ has a lot in common with these
tensor algebras.  However, since the semicrossed product is defined as
the universal operator algebra for a family of covariant representations,
one needs to prove a dilation theorem in order to decide whether or not this algebra
sits inside a Cuntz--Pimsner algebra completely isometrically.
This is, in fact, where all of the difficulty lies.

It is perhaps worth mentioning that classical counterexamples in dilation
theory point to the difficulties that might arise in general.  The bidisk
algebra $A(\bD^2)$ sits inside $\rC(\bT^2)$, which is its C*-envelope
by Ando's theorem.  Consider the identity automorphism $\id$.
Ando's theorem also shows that the completely contractive representations
of $A(\bD^2)$ are determined by an arbitrary pair $T_1,T_2$ of
commuting contractions.  A covariant representation of $(A(\bD^2),\id)$
is given by such a pair and a third contraction $T_3$ which commutes
with $T_1$ and $T_2$.  If it were true that the C*-envelope of this system
was $\rC(\bT^2) \times_\id \bZ \simeq \rC(\bT^3)$, then it would be true
that every commuting triple of contractions has a unitary dilation.
This was disproven in a famous paper by Varopolous.
See Paulsen's book \cite[Chapter 5]{Pau} for a treatment of these topics.
Thus when such unitary dilations are possible, we must see this as
an important but special phenomenon.
 
\section{Preliminaries} \label{S:prel}

Consider the left regular representation $\lambda$ of the
free semigroup $\Fn$ acting on Fock space, $\Fock$.
Let $L_i = \lambda(i)$.
The non-commutative disc algebra $\fA_n$, for $n \ge 2$,  is the nonself-adjoint
unital operator algebra generated by $L_1,\dots,L_n$.
It sits as a subalgebra of the Cuntz--Toeplitz C*-algebra $\E_n = \ca(\{L_1,\dots,L_n\})$.
However the quotient map onto the Cuntz algebra $\O_n$ is completely isometric on $\fA_n$.
So $\fA_n$ may be considered as the subalgebra of $\O_n$ generated by the
standard generators $\fs_1,\dots,\fs_n$.
Moreover the operator algebra generated by any $n$-tuple of isometries
$S_1,\dots,S_n$ with pairwise orthogonal ranges is completely isometrically
isomorphic to $\fA_n$.
These algebras were introduced by Popesu \cite{Pop2} as a natural multivariable
generalization of disc algebra $A(\bD)$.
The Frahzo--Bunce dilation Theorem \cite{Fr,Bun,Pop1} shows that any row contractive $n$-tuple
$T=[T_1, \dots, T_n]$ dilates to an $n$-tuple of isometries $S = [S_1, \dots, S_n]$
with pairwise orthogonal range. Hence given any such $n$-tuple $T$, there is a unique
completely contractive homomorphism of $\fA_n$ onto the
algebra $\A(T_1,\dots,T_n)$ taking generators to generators.
Popescu \cite{Pop2} used this to establish a natural analogue of the von Neumann inequality
for row contractive $n$-tuples.

If $\phi$ is an automorphism for an operator algebra $\A$, then a contractive
(resp.\ isometric or unitary) covariant representation for $(\A, \phi)$ consists of
a completely contractive (resp.\ completely isometric) representation
$\pi: \A \rightarrow B(\H)$ and a contraction (resp.\ isometry or unitary) $ U \in B(\H)$ so that
$\pi(A)U= U \pi(\phi(A))$.
If $\A$ happens to be a $\ca$-algebra, then completely contractive maps are
$*$-homomorphisms.

Each element $A \in \fA_n$ determines a function $\hat A$ on the character space,
and this is a bounded holomorphic function on $\bB_n$ which extends to a
continuous function on $\ol{\bB}_n$.
An automorphism $\phi$ of $\fA_n$ induce an automorphism $\hat\phi$ of the character space
$\bB_n$ via $\hat A(\hat \phi (z)) = \widehat{\phi(A)}(z)$.
The map $\hat\phi$ is biholomorphic, and thus is a conformal automorphism \cite{DP1}.
These maps are given by fractional linear transformations (c.f.  Rudin \cite{Rud}).

Each of these conformal maps is induced by a unitarily implemented automorphism
of $\E_n$ which fixes the subalgebra $\fA_n$.  Indeed, Voiculescu \cite{V} constructs
a unitary representation of the Lie group $U(n,1)$ which contains the scalar unitaries, and
$U(n,1)/\bT \simeq \Aut(\bB_n)$, such that $\ad U$ implements the corresponding
automorphism.  In \cite{DP2}, the first author and Pitts study the
automorphism of the weak operator closed algebra $\fL_n = \ol{\fA}_n^{\wot}$.
The case of $\fA_n$ is similar but more elementary.
See Popescu \cite{Pop_auto} for another proof.

\begin{defn}   \label{def:sem}
Let $\Aut(\fA_n)$ denote the group of completely isometric automorphisms
of $\fA_n$, and let $\phi \in \Aut(\fA_n)$.
A \textit{covariant representation} $(\pi, K)$ of $(\fA_n,\phi)$
is a completely contractive representation
$\pi$ of $\fA_n$ on a Hilbert space $\H$ and a contraction $K \in \B(\H)$ so that
\[ \pi(A) K = K \pi(\phi(A)) \qforal A \in \fA_n .\]
The \textit{semicrossed product} $\fA_n \times_\phi \bZ^+$ is the
universal operator algebra generated by a copy of $\fA_n$ and a contraction $\fu$
so that $A\fu = \fu \phi(A)$ for all $A \in \fA_n$.
\end{defn}

In other words, $\fA_n \times_\phi \bZ^+$ is the operator algebra generated by a
(completely isometric) copy of $\fA_n$ and a contraction $\fu$ with the property
that every covariant representation of $(\fA_n,\phi)$ induces
a completely contractive representation $\pi \times K$ of
$\fA_n \times_\phi \bZ^+$ on $\H$, which on polynomials is defined as
\[
 (\pi \times K) \big( \sum \fu^n A_n \big) = \sum K^n \pi(A_n).
\]
The norm may be defined by as the supremum over all covariant representations:
\[
  \big\| \sum \fu^n A_n \big\| = \sup_{(\pi,K)} \big\| (\pi \times K)\big(\sum \fu^n A_n \big) \big\| .
\]

A completely contractive representation of $\fA_n$ sends the generators
$L = [L_1\ \dots\ L_n]$ to a row contraction $A = [A_1\ \dots \ A_n]$.
Conversely, the Frahzo--Bunce dilation theorem \cite{Fr,Bun} shows that
any row contraction dilates to a row isometry.  Thus by Popescu's von Neumann inequality
\cite{Pop2}, there is a completely contractive representation $\pi$ of $\fA_n$
with $\pi(L_i)=A_i$ for $1 \le i \le n$.
If we take $\pi=\id_{\fA_n}$ and $U=0$, we see that the imbedding of $\fA_n$ into
$\fA_n \times_\phi \bZ^+$ is completely isometric.

In \cite{DKn}, we identified several such representations which are worth repeating.

\begin{eg} \label{basiceg}
For any $\phi\in\Aut(\fA_n)$, consider the unitary $U_\phi$ constructed by
Voiculescu \cite{V} on the Fock space $\Fock$ which implements the action of
$\phi$ on the Cuntz--Toeplitz $\ca$-algebra $\E_n$ by $U_\phi^*AU_\phi = \phi(A)$,
and fixes $\fA_n$.
For any $\lambda \in \bT$, this provides a covariant pair $(\id_{\E_n}, \lambda U_\phi)$
for $(\E_n, \phi)$.
Therefore it produces a representation $\id_{\E_n}\times\lambda U_\phi$ of $\E_n\times_\phi \bZ$.
Since $\fA_n$ is invariant for $\ad U_\phi$, this provides a covariant representation
$(\id, \lambda U_\phi)$ of $(\fA_n , \phi)$ by restriction.
This yields a representation $\id \times \lambda \hat{U}_\phi$ of $\fA_n \times_\phi \bZ^+$ which is
completely isometric on $\fA_n$ and $\lambda U_\phi$ is unitary.

Similarly, by taking a quotient by $\K(\Fock)$, the compact operators on $\Fock$, we obtain a covariant
representation $(\pi_{\O_n}, \lambda \hat{U}_\phi)$ for $(\O_n, \phi)$ and therefore representations for both
$\O_n \times_\phi \bZ$ and $\fA_n \times_\phi \bZ^+$, inside
the Calkin algebra, which we denote as $\pi_{\O_n}\times \lambda \hat{U}_\phi$.
\end{eg}

\begin{eg}
Let $\pi$ be any completely contractive representation of $\fA_n$ on a Hilbert space $\H$.
Define $\tilde\pi$ on $\H \otimes \ltwo$ by
\[ \tilde\pi(a) = \sum\strut_{k\ge0}^\oplus \pi\phi^k(a) \qand U = I_\H \otimes S \]
where $S$ is the unilateral shift.  This is easily seen to yield a completely contractive
representation $\tilde\pi$ of $\fA_n$ and a contraction $U$ so that
$\tilde\pi \times_\phi U$ yields a representation of $\fA_n \times_\phi \bZ^+$.

These are called \textit{orbit representations}.
When $\A$ is a C*-algebra, Peters \cite{Pet} showed that the direct sum of all
orbit representations $\tilde\pi\times_\phi U$, as $\pi$ runs over the $*$-representations
of $\A$, yields a completely isometric representation of $\A \times_\phi \bZ^+$.
For general operator algebras, this is not the case.
\end{eg}

\begin{eg} Let $\phi \in \Aut \fA_n$.
Consider the non-commutative disc algebra $\fA_{n+1}$ acting on the Fock space
$\ell^2 (\mathbb{F}_{n+1}^{+})$ and define an ideal
\[ \fJ = \big\lip L_iL_{n+1} -L_{n+1}\phi(L_i) : 1 \le i \le n \big\rip .\]
The \wot-closure $\ol{\fJ}$ of $\fJ$ is an ideal of $\fL_{n+1}$, and these ideals were studied
in \cite{DP2,DP3}.  In particular, it is shown in \cite{DP2} that $\ol{\fJ}$ is determined by
its range, which is a subspace invariant for both $\fL_{n+1}$ and its commutant $\fR_{n+1}$.
Then in \cite{DP3}, it is shown that  $\fL_{n+1}/\ol{\fJ}$ is completely isometrically
isomorphic to the compression to $\M_\phi = \ran(\fJ)^\perp$.  Since
\[
  \ol{\ran{\fJ}} = \spn \big\{  A(L_iL_{n+1} -L_{n+1}\phi(L_i))\ell^2 (\mathbb{F}_{n+1}^{+}):  i=1, 2, \dots, n \big\}
\]
is evidently orthogonal to $\xi_\mt$, we see that $\M_\phi$ is non-empty.
The compression of $\fA_n$ to $ \M_{\phi}$ is a completely contractive homomorphism $\rho$,
and the compression $B$ of $L_{n+1}$ is also a contraction.  Therefore
$(\rho, B)$ is a covariant representation of $(\fA_n,\phi)$, and thus
determines a completely contractive representation $\fA_n \times_\phi \bZ^+$.
\end{eg}

\begin{eg} Any representation of $\fA_n$ produces a
representation of $\fA_n \times_\phi \bZ^+$
by simply taking $U=0$.
In \cite{DKn}, we contructed various finite dimensional
representations of $\fA_n \times_\phi \bZ^+$ which allowed us to
classify them as algebras.
\end{eg}

\section{Unitary Covariant Representations} \label{S:dil}

The purpose of this section is to show that contractive representation of our
covariant system always dilate to a unitary covariant system.
The proof requires a number of known dilation theorems.

We call a representation $\pi$ of an algebra $\A$
on a Hilbert space $\K \supset \H$ an \textit{extension} of a
representation $\sigma$ of $\A$ on $\H$ if $\H$ is invariant for
$\pi(\A)$ and $\pi(A)|_\H = \sigma(A)$ for $A \in \A$; i.e., $\pi(A) \simeq
\begin{bmatrix} \sigma(A) & * \\ 0 & * \end{bmatrix}$.
Likewise, $\pi$ is a \textit{co-extension} of $\sigma$ if $\H$ is
co-invariant for $\pi(\A)$ and $P_\H \pi(A)|_\H = \sigma(A)$
for $A \in \A$; i.e., $\pi(A) \simeq
\begin{bmatrix} * & * \\ 0 & \sigma(A) \end{bmatrix}$.
Finally, we say that $\pi$ is a \textit{dilation} of $\sigma$ if
$P_\H \pi(A)|_\H = \sigma(A)$  for $A \in \A$. By a result of
Sarason \cite{Sar}, $\H$ is semi-invariant and so $\pi(A) \simeq
\begin{bmatrix}*&*&*\\ 0& \sigma(A) & * \\ 0 & 0 & * \end{bmatrix}$.

The main result of this section will be established by a sequence of lemmas.

\begin{thm}    \label{main1}
Let $\phi$ be an isometric
automorphism of the non-comm\-ut\-ative disc algebra $\fA_n$, $n \geq 2$.
Then, any contractive covariant representation of $(\fA_n, \phi)$ dilates to
a unitary covariant
representation of $(\E_n, \phi)$, where $\E_n$ denotes the Cuntz-Toeplitz $\ca$-algebra.
\end{thm}

According to the commutant lifting Theorem of Frazho \cite{Fr2}
 and Popescu \cite{Pop1}, if $S,T$ are row contractions
and $K$ another contraction intertwining them,
 i.e.,  $S^{(i)}K=KT^{(i)}$ for all $1 \le i \le n$, then $K$ co-extends
to a contraction $K'$ that intertwines the minimal isometric dilations $V_S$ and $V_T$
of $S$ and $T$ respectively. A dual result can be obtained from a recent dilation of
Solel \cite{Sol1}. (See also \cite{DPY}.)
Solel's result says that if $S$ and $ T$ are as above, then
we can co-extend the contractions $S, T$ and $K$
to isometries $W_S , W_T$ and $W_K$, which still satisfy
$W_{S}^{(i)}W_K=W_KW_{T}^{(i)}$.
This leads to the following.

\begin{lem} \label{fromSolel}
Assume that $S=[S^{(1)}, \dots, S^{(n)}]$, $T=[T^{(1)}, \dots, T^{(n)}]$
are row contractions and $K$ is a contraction on $\H$ so that
\[
S^{(i)}K=KT^{(i)}, \quad 1\leq i\leq n.
\]
Let $V$ be an isometric dilation of $K$, acting on a Hilbert space $\H'$.
Then there exist row contractions $S'=[S'^{(1)}, \dots, S'^{(n)}]$, $T'=[T'^{(1)}, \dots, T'^{(n)}]$
on $\H'$, which co-extend $S$ and $T$ respectively,
and satisfy
\[
S'^{(i)}V =V T'^{(i)}, \quad 1\leq i\leq n.
\]
\end{lem}

\begin{proof}
Let $V_K$ be the minimal isometric dilation of $K$.
Then we can decompose $V \simeq V_K \oplus V'$.
So if we can dilate $S$ and $T$ to $S'$ and $T'$ intertwining $V_K$,
 then we can extend further to all of $\H'$ by adding zero summands to $S'$ and $T'$.

By Solel's result, $S, T$ and $K$ co-extend to intertwining isometries
$W_S, W_T$  and $W_K$, acting on a Hilbert space $\H^{\prime \prime}$.
Let $\H' = \bigvee_{j\ge0} W_K^j \H$ be the smallest invariant subspace
of $W_K$ containing $\H$.
Clearly, $\H'$ is reducing for $W_K$ and the restriction of $W_K$
on $\H'$ is (unitarily equivalent to) the minimal dilation $V_K$.
The result now follows by setting
\[
 S'^{(i)} =P_{\H'}S^{(i)}\mid_{\H'} \qand
 T'^{(i)} =P_{\H'}T^{(i)}\mid_{\H'}
 \qfor 1 \le i \le n .\qedhere
\]
\end{proof}

\begin{lem} \label{nextstep}
Let $\phi \in \Aut (\fA_n)$ and let $A=[A^{(1)}, \dots A^{(n)}]$ and $K$ be contractions
satisfying the covariance relations
$ A^{(i)}K=K\phi(A)^{(i)} $ for $1\leq i \leq n.$
Then there exist isometries $T_A= [T_{A}^{(1)}\ \dots\ T_{A}^{(n)}]$ and $T_K$,
dilating $A$ and $K$ respectively, so that
\[
 T_{A}^{(i)} T_K = T_K \phi(T_{A})^{(i)}  \qfor 1\leq i \leq n.
\]
\end{lem}

\begin{proof} Notice that if $V_A$ is the minimal isometric dilation of $A$,
then $\phi(V_A)$ is the minimal isometric dilation of
$\phi(A)= [\phi(A)^{(1})\ \dots\ \phi(A)^{(n)}]$.
Therefore, by applying commutant lifting to the covariance relations,
we obtain a contraction $K_1$ on a Hilbert space $\H_1$, satisfying
\[
 V_A^{(i)}K_1=K_1\phi(V_A)^{(i)}
\]

Let $S_{K_1}$ be the Schaeffer dilation of $K_1$ on $\H_1^{(\infty)}$ by
\[
 S_{K_1} \simeq
  \begin{bmatrix}
   K_1 &  0 &  0 & 0 &  \dots \\
   D_{K_1}  &0 &  0  & 0 &  \dots \\
   0  &  I  & 0 &  0 & \dots \\
   0  &  0  & I & 0  &  \dots   \\
   \vdots & \vdots & \vdots & \vdots  & \ddots
   \end{bmatrix}
\]
where $D_{K_1} = (I-K_1^{*}K_1)^{1/2}$.
We apply Lemma \ref{fromSolel}, with $S=V_A$, $T=\phi(V_A)$, $K=K_1$
and its isometric dilation $S_{K_1}$
to obtain row contractions $\hat{A_1}$ and $\hat{B_1}$, which
co-extend $V_A$ and $\phi(V_A)$, and satisfy
\[ 
 \hat{A}_1^{(i)} S_{K_1} = S_{K_1}\hat{B}_1^{(i)} \qfor 1 \le i \le  n.
\] 
Because $V_a^{(i)}$ are already isometries,
these dilations have the form
\[
 \hat{A}_1^{(i)} =
 \begin{bmatrix} V_A^{(i)} &  0  \\ 0  & [X^{(i)}_{jk}]_{j,k\ge1}  \end{bmatrix}
 \qand
 \hat{B}_1^{(i)} =
 \begin{bmatrix} \phi(V_A^{(i)}) &  0  \\ 0  & [Y^{(i)}_{jk}]_{j,k\ge1}  \end{bmatrix}
\]
By comparing $(2,1)$-entries in the covariance relation,  we obtain
\[
 X^{(i)}_{11} D_{K_1} = D_{K_1} \phi(V_A^{(i)}) \qfor 1\leq i \leq n .
\]

For simplicity, write $X_i = X^{(i)}_{11}$.
Note that $X=[X_1\ \dots, \ X_n]$ is a row contraction, and so $\phi(X)$ is meaningful.
For $1\leq i \leq n$, we now define
\[
A_{1}^{(i)}=
 \begin{bmatrix}
   V_A^{(i)} &  0 &  0 & 0 &  \dots \\
   0  & X_i &  0  & 0 &  \dots \\
   0  &  0  & \phi(X)^{(i)} &  0 & \dots \\
   0  &  0  &0     & (\phi \circ \phi)(X)^{(i)}  &  \dots   \\
   \vdots & \vdots & \vdots & \vdots  & \ddots
   \end{bmatrix}
\]
We obtain a row contraction $A_1=[A_{1}^{(1)}\ \dots\  A_{1}^{(n)} ]$
on a Hilbert space $\H_2$ so that
\[
   A_{1}^{(i)} S_{K_1} = S_{K_1} \phi(A_1)^{(i)}  \qfor 1\leq i \leq n.
\]

Continuing in this fashion, we obtain a sequence
\[
     (A, K), \,(V_A, K_1),\, (A_1, S_{K_1}), \,(V_{A_1}, K_2), \, ( A_2 , S_{K_2}) \dots
\]
of pairs of operators acting on Hilbert spaces
$
  \H \subseteq \H_1 \subseteq \H_2 \dots,
$
co-extending $A$ and $K$ and satisfying the covariance relations.
Let $\H = \bigvee_j \H_j$, and consider these pairs of operators as
acting on $\H$ by extending them to be zero on the complement.
Let
\[
 T_A = \sotlim A_j = \sotlim V_{A_j}
\]
and
\[
 T_K = \sotlim S_{K_j} = \sotlim K_j .
\]
These limits evidently exist as in each case, one of the sequences consists of
isometries which decompose as infinite direct sums.
In particular, $T_A$ is a row isometry and $T_K$ is an isometry.
Multiplication is \sot-continuous on the ball,
hence the covariance relations hold in the limit.
\end{proof}

We now extend this to a unitary representation.
The proof uses the ``one step extension'' technique.

\begin{lem} \label{finalstep}
Let $\phi \in \Aut (\fA_n)$ and let $S=[S^{(1)}, \dots S^{(n)}]$ be a row isometry
and let $V$ be an isometry acting on a Hilbert space $\H$ and
satisfying the covariance relations
\[
  S^{(i)}V=V\phi(S)^{(i)} \qfor 1\leq i \leq n .
\]
Then there exist a row isometry $\tilde S = [\tilde S^{(1)}\ \dots\ \tilde S^{(n)}]$ and
an isometry $\tilde V$, acting on a Hilbert space $\tilde\H \supset\H$,
extending $S$ and $V$ respectively and satisfying
\begin{itemize}
 \item[(i)] $\tilde S^{(i)} \tilde V = \tilde V \phi(\tilde S)^{(i)} \qfor 1\leq i \leq n$
 \item[(ii)] $\tilde V (\tilde\H) = \H$.
\end{itemize}
\end{lem}

\begin{proof}
Let $\K= (I-VV^*)\H$ and set $\H' = \H \oplus \K$.
Define a unitary operator $U \in B(\H \oplus \K, \H)$ by
\[
 U(x, y) = Vx+y \qfor x \in \H \AND y \in \K.
\]
Set
\[
  \tilde V = U^*VU  \qand
  \tilde S^{(i)}= U^*\phi^{-1}(S)^{(i)} U \qfor i=1, \dots , n.
\]
Notice that $U^*(x)= (Vx\,,(I-VV^*)x)$ and so
\[
 \tilde V(x,y) = (Vx+y,0) \qfor x \in \H \AND y \in \K.
\]
Therefore $\tilde V$ extends $V$ and maps $\H'$ onto $\H$.

To show that $\tilde S$ extends $S$, note that the covariance relations
imply that
\[
  \phi^{-1}(S)^{(i)}V = VS^{(i)}  \qfor 1\leq i \leq n .
\]
Hence, for any $x \in \H$ we have for $x \in \H$,
\begin{align*}
 \phi^{-1}(S)^{(i)}U(x,0)&=\phi^{-1}(S)^{(i)}Vx   
                            =VS^{(i)}x
                            =U (S^{(i)}x, 0) .
\end{align*}
Hence $\tilde S^{(i)}|_\H = S^{(i)}$.
Finally, this same calculation shows that
\[
 \tilde S^{(i)} \tilde V = U^*\phi^{-1}(S)^{(i)} V U
 = U^* VS^{(i)} U = \tilde V  \phi(\tilde S)^{(i)} .\qedhere
\]
\end{proof}

We can now complete the proof of the main result.

\begin{proof}[\textbf{\em Proof of Theorem~\ref{main1}}]
Let $A=[A^{(1)}, \dots A^{(n)}]$ and $K$ be contractions on a Hilbert space $\H$
satisfying
\[
A^{(i)}K=K\phi(A)^{(i)}, \quad 1\leq i \leq n.
\]
Using Lemma \ref{nextstep}, we dilate $A$ and $K$ to isometries
$S$ and $V$  satisfying $S^{(i)}U=U\phi(S)^{(i)}$ for $1\leq i \leq n$.
Making repeated use of Lemma \ref{finalstep}, we now produce a sequence
$\{(S_j,V_j)\}_{j=1}^{\infty}$ of extensions consisting of a row isometry
$S_j$ extending $S_{j-1}$ and an isometry $V_j$ extending $V_{j-1}$,
acting on an increasing sequence of Hilbert spaces $\H_j$, which
satisfy the covariance relations and have $V_j \H_j = \H_{j-1}$.
If we set $S = \sotlim S_j$ and $U = \sotlim V_j$, then condition
(ii) in Lemma \ref{finalstep} implies that $U$ is a unitary while (i)
shows that $S$ and $U$ satisfy the covariance relations.
\end{proof}

\section{Maximal Covariant Representations
and the $\ca$-envelope of $\fA_n \times_{\phi} \bZ^+$.} \label{S:full}

There is a question left open in Theorem~\ref{main1}, which is whether the
row isometry in the unitary dilation generates the Cuntz algebra or the
Cuntz--Toeplitz algebra.  It is not hard to see that in the former case, there
is no sensible way to dilate further.  But in the Cuntz--Toeplitz case, there
is a gap, since $\sum_{i=1}^n S_iS_i^* < I$, that may allow a proper dilation.
In fact this occurs, and in this section we will deal with this issue.

The Dritschel--McCullough proof \cite{DMc} of Hamana's Theorem \cite{Ham}
proving the existence of Arveson's C*-envelope \cite{Arv1} is based on the
notion of a \textit{maximal representation}.  This is a completely contractive
representation $\rho$ of an  operator algebra with the property that the only
(completely contractive) dilations have the form $\rho \oplus \sigma$.
They establish that every representation dilates to a maximal one, and
that maximal representations extend to $*$-representations of the
C*-envelope.  In this manner, they were able to establish the existence
of the C*-envelope without taking Hamana's route via the injective envelope.
The upshot for dilation theory is to focus attention on maximal dilations.

In our case, Theorem~\ref{main1} shows that the maximal dilations
must send the generators of $\fA_n$ to a row isometry $S$ and the operator
implementing the automorphism must be unitary.  In the case when
this representation is of Cuntz type, meaning that $SS^* = I$, it is evident
that this representation is maximal.  So we are left to deal with the other case.

We first show that the Wold decomposition of $S$ decomposes $U$ as well.
Recall that the Wold decomposition uniquely splits the
Hilbert space into $\H = \H_0 \oplus \H_1$
so that $S_i|_{\H_0} \simeq L_i^{(\alpha)}$ is pure,
and $T_i := S_i|_{\H_1}$ has Cuntz type.

\begin{lem}
 Suppose $S = [S_1\ \dots \ S_n]$ is a row isometry and $U$ is a unitary on a Hilbert
 space $\H$ satisfying the covariance relations $S_iU = U \phi(S_i)$ for $1 \le i \le n$.
 Then the Wold decomposition reduces $U$, thereby decomposing the representation
 of $\fA_n \times_\phi \bZ^+$ into a pure part and a Cuntz part.
\end{lem}

\begin{proof}
Let $\sigma$ be the representation of $\fA_n \times_\phi \bZ^+$
with $\sigma(L_i) = S_i$ and $\sigma(\fu) = U$.
Let $M = \ran(I-SS^*) = \ran(I-\sum_{i=1}^n S_iS_i^*)$.
Then $\H_0 = \ol{ \sigma(\fA_n) M}$.
Now $[\phi(S_1)\ \dots\ \phi(S_n)]$ is also a row isometry, and
we let $N = \ran(I - \sum_{i=1}^n \phi(S_i) \phi(S_i)^*)$.
Since $S_i \simeq L_i^{(\alpha)} \oplus T_i$, we see that
$\phi(S_i) = \phi(L_i)^{(\alpha)} \oplus \phi(T_i)$.
Thus the Wold decomposition of $\phi(S)$ decomposes $\H$
in the same way as $S$.  Therefore $\ol{\sigma(\fA_n) N} = \H_0$.

Now we use the fact that $U$ implements $\phi$ to see that
$N=UM$ and so
\[ U \H_0 = U \ol{\sigma(\fA_n) M} = \ol{ \sigma(\fA_n) U M} = \ol{\sigma(\fA_n) N} = \H_0 .\]
Therefore $\H_0$ reduces $U$ as claimed.
\end{proof}

Next we show how $\phi$ is implemented on $\H_0$.

\begin{lem}
 Let $\phi \in \Aut(\fA_n)$ and let $U_\phi$ be the Voiculescu unitary on $\Fock$
 which implements $\phi$.  Then the only unitaries on $\B(\Fock^{(\alpha)}$
 which implement $\phi$ on $\fA_n^{(\alpha)}$ have the form $U_\phi\otimes W$.
\end{lem}

\begin{proof}
Clearly $U_\phi \otimes I_\alpha$ implements $\phi$.  If $V$ is another unitary
implementing $\phi$ on $\fA_n^{(\alpha)}$, then $(U_\phi^*\otimes I_\alpha)V$
commutes with $\fA_n^{(\alpha)}$.  By Fuglede's Theorem, it commutes with
$\ca(\fA_n^{(\alpha)})^{\prime\prime} = \B(\Fock) \otimes \bC I_\alpha$.  Therefore
it lies in $\ca(\fA_n^{(\alpha)})' = \bC I_\Fock \otimes \B(\H)$ where $\dim\H = \alpha$,
say $(U_\phi^*\otimes I_\alpha)V = I \otimes W$.
\end{proof}

Now $W$ is unitary, and so has a spectral resolution.
So essentially every pure representation $(\id^{(\alpha)},U)$ of
the covariance relations is a direct integral of the representations
$(\id,\lambda U_\phi)$ as $\lambda$ runs over the unit circle $\bT$.
Thus it suffices to show how to dilate $(\id,U_\phi)$ to a Cuntz type
unitary dilation.

To accomplish this, we need to consider the map $\hat\phi$ in $\Aut(\bB_n)$.
We refer to \cite[Chapter 2]{Rud} for details.  We distinguish two cases.
In the first case,  $\hat\phi$ has a fixed point inside $\bB_n$.
Because $\Aut(\bB_n)$ acts transitively on $\bB_n$,  $\hat\phi$ is
biholomorphically conjugate to a map which fixes $0$.  Such an
equivalence yields a completely isometric isomorphism of the
semi-crossed products.  So we may assume that  $\hat\phi(0) = 0$
without loss of generality.  But then  $\hat\phi$ is a unitary matrix $U_0\in\U(n)$,
$\phi$ is the gauge automorphism it induces, and
\[ U_\phi = \sum\strut_{i\ge0}^\oplus U_0^{\otimes i} .\]
In the second case,  $\hat\phi$ fixes one or two points on the unit sphere.
Again $\Aut(\bB_n)$ acts transitively on the sphere, so we may suppose
that $e_1 = (1,0,\dots,0)$ is a fixed point.
We will deal with these two cases separately.

In both cases, we will dilate to atomic representations of the Cuntz algebra.
These are $*$-representations in which the generators permute an orthonormal
basis up to scalar multiples.  These representations were defined and
classified in \cite{DP1}. In the first case, we use representations of inductive type.
Beginning with an infinite tail, i.e., an infinite word $\bx = i_1i_2\dots$ in the alphabet
$\{1,\dots,n\}$, define a sequence of Hilbert spaces $\H_k$, for $k \ge 0$, as follows.
Each $\H_k$ naturally identified with Fock space $\Fock$, and this determines
the action of $\Fn$ on $\H_k$ by the left regular representation, which
extends to a $*$-representation $\lambda_k$ of the Cuntz--Toeplitz algebra $\E_n$.
Imbed $\H_{k-1}$ into $\H_k$ by the isometry $V_k \xi^{k-1}_w = \xi^k_{wi_k}$,
where with basis $\{\xi^k_w : w \in \Fn\}$ is the standard basis for $\H_k$.
Effectively, $V_k$ is unitarily equivalent to $R_{i_k}$, the right multiplication
operator by the symbol $i_k$.  Since this lies in the commutant of the left
regular representation, it is evident that $V_k$ intertwines $\lambda_{k-1}$ and $\lambda_k$.
The inductive limit of these representations, denoted $\lambda_\bx$, on the
Hilbert space $\H_\bx = \lim \H_k$,
is a $*$-representation  of $\E_n$ onto the Cuntz algebra because in the limit,
the sum of the ranges of $\lambda_\bx(\fs_i)$ for $1 \le i \le n$ is the whole space.

\begin{thm} \label{T:Cuntz dilate 1}
Let $\phi\in\Aut(\fA_n)$ such that $\hat\phi$ has a fixed point in $\bB_n$.
Then $(\id,U_\phi)$ has a unitary dilation of Cuntz type.
\end{thm}

\begin{proof}
As noted before the proof, $\phi$ is biholomorphically conjugate to an automorphism
which fixes the origin, and hence is a gauge automorphism.  So we start by assuming
that $\phi$ has this form;
so $\phi$ is determined by the unitary $\hat\phi = U_0$ on $\spn\{\xi_i : 1 \le i \le n\}$.

Since unitary matrices are diagonalizable, the map $\hat\phi$ is biholomorphically
conjugate to a diagonal unitary.
Thus it suffices to assume that $U_0$ is diagonal,
say $U_0 \xi_i = \mu_i \xi_i$ for scalars $\mu_i \in \bT$.
Let us write $\mu=(\mu_1,\dots,\mu_n) \in \bT^n$.
It is easy to verify that $U_\phi$ is the diagonal operator $U_\phi \xi_w = w(\mu) \xi_w$.

Now let $\bx = i_1i_2\dots$ be any infinite tail, and consider the construction
indicated before this proof.  Set $x_k = i_1i_2\dots i_k$ for $k \ge 1$.
Define unitaries $V_k$ on $\H_k$ by $V_k \xi^k_w = \ol{x_k(\mu)} w(\mu) \xi^k_w$.
It is easy to see that since this is a scalar multiple of $U_\phi$, conjugation by
$V_k$ implements $\phi$ on $\lambda_k(\E_n)$.  Moreover the scalar $\ol{x_k(\mu)}$
is chosen so that $V_k|_{\H_{k-1}} = V_{k-1}$ for $k \ge 1$.  Thus the inductive limit
yields the representation $\sigma_\bx$ and a unitary operator $V$ on $\H_\bx$
implementing $\phi$.  Thus $(\sigma_\bx,V)$ is the desired dilation.

Note that the discussion prior to the theorem implies now that \textit{any} representation of
$\fA_n\times_{\phi}\bZ^+$ dilates to a Cuntz-type representation, provided that $\phi$ is a gauge automorphism.

In case of an arbitrary $\phi$, we want to prove the existence of a Cuntz-type dilation for
$(\id, U_{\phi})$.
As in the discussion prior to the theorem, there exists a biholomorphic automorphism
$\alpha$ and  a gauge automorphism $\phi'$ so that $\phi' \circ \alpha = \alpha \circ \phi$.
By the previous paragraph, $(\id,\phi')$ has a unitary Cuntz dilation $(\sigma_\bx,V)$.
We claim that $(\sigma_\bx\alpha,V)$ provides a unitary Cuntz dilation of $(\id,\phi)$.
It suffices to verify the covariance relations:
\[
 \sigma_\bx\alpha(A)W = W \sigma_\bx(\phi'(\alpha(A))) = W \sigma_\bx\alpha (\phi(A)) .\qedhere
\]
\end{proof}

For the second case, we use a special case of the ring representations \cite{DP1}.
Let $\H_j = \Fock$ with basis $\{\xi^j_w : w \in \Fn\}$ for $2 \le j \le n$.
Let $\H = \bC \zeta \oplus \sum_{j=2}^n \oplus  \H_j$.
Let $\sigma_1$ denote the representation determined by
\begin{gather*}
 \sigma_1(L_1) \zeta = \zeta,\qquad
 \sigma_1(L_j)\zeta = \xi^j_\mt \quad\FOR 2 \le j \le n, \\
 \sigma_1(L_i) \xi^j_w = \xi^j_{iw} \quad\FOR 1\le i \le n,\ 2 \le j \le n,\ w \in \Fn.
\end{gather*}
This is evidently a Cuntz representation.  Moreover, $\bC \zeta$ is coinvariant and
thus the compression to $\bC \zeta$ is a multiplicative functional $\psi$ such that
\[ \psi(A) = \ip{ \sigma_{1,\mu}(L_i)\zeta,\zeta} = \delta_{i1} = \hat L_i (e_1) .\]
Hence $\ip{ \sigma_1(A) \zeta,\zeta} = \hat A(e_1)$ for all $A \in \fA_n$.

\begin{thm} \label{T:Cuntz dilate 2}
Let $\phi\in\Aut(\fA_n)$ such that $\hat\phi$ has a fixed point on
the boundary of $\ol{\bB}_n$.
Then $(\id,U_\phi)$ has a unitary dilation of Cuntz type.
\end{thm}

\begin{proof}
As in the previous proof, we may
suppose that $\hat\phi$ has $e_1$ as a fixed point.

First we show that $\phi$ is unitarily implemented on $(\sigma_1,\H)$.
Let $\psi(A) = \ip{\sigma_1(A) \zeta,\zeta} = \hat A(e_1)$.
Define a unitary $W = 1 \oplus U_\phi^{(n-1)}$, and consider $S_i = W^* \sigma_1(L_i)W$.
Then
\[ S_i|_{(\bC\zeta)^\perp} = \phi(L_i)^{(n-1)} \qfor 1 \le i \le n .\]
Also
\begin{align*}
 \ip{S_i \zeta,\zeta} &= \psi(L_i) = \delta_{i1}  
 = \hat L_i(e_1) = \hat L_i \hat\phi(e_1) = \widehat{\phi(L_i)}(e_1) .
\end{align*}
In particular,
\[
 \ip{S_1\zeta,\zeta} = 1 =
 \ip{\sigma_1(\phi(L_1))\zeta,\zeta}   .
\]
Since both $S_1$ and $\sigma_1(\phi(L_1)) = \phi(\sigma_1(L_1))$ are isometries,
we conclude that $S_1\zeta = \zeta = \sigma_1(\phi(L_1)) \zeta$.
Both agree with $\phi(L_1)^{(n-1)}$ on $(\bC \zeta)^\perp$, and
therefore
\[ S_1 = 1 \oplus \phi(L_1)^{(n-1)} = \sigma_1(\phi(L_1)) .\]

On the other hand,  $S_j\zeta$ is orthogonal to $\zeta$ for $2 \le j \le n$.
Because these are isometries with pairwise orthogonal ranges, $S_j\zeta$
is also orthogonal to
\[
 \Big( \sum_{i=1}^n \phi(L_i) \Fock \Big)^{(n-1)} =
 \Big( U_\phi^* \sum_{i=1}^n L_i U_\phi\Fock \Big)^{(n-1)} =
 \big( (\bC \nu)^\perp  \big)^{(n-1)}
\]
where $\nu = U_\phi^*\xi_\mt$.
Observe that exactly the same is true for the isometries $\sigma_1(\phi(L_j))$
because $\sigma_1(\phi(L_j))|_{(\bC\zeta)^\perp} = \phi(L_i)^{(n-1)}$ also.
Therefore there is a unitary $V$ on $( \bC \nu )^{(n-1)}$ so that
\[ VS_j \zeta = \sigma_1(\phi(L_j))\zeta \qfor 2 \le j \le n .\]
Considering $V$ as an operator on $\bC^{n-1}$, we define
$V' = I_{\Fock} \otimes V$ in the commutant of $\fA_n^{(n-1)}$
extending $V$ to all of $(\bC\zeta)^\perp$.
Define $W' = (1 \oplus V')^*W$.
Then
\[ W^{\prime *} \sigma_1(L_i) W' = \sigma_1(\phi(L_i)) \qfor 1 \le i \le n .\]

Pick a unit eigenvector $y \in \bC^{n-1}$ for the unitary matrix $V$,
say $Vy = \beta y$.
Then $\H = \Fock \otimes \bC y$ is an invariant subspace for
$\sigma_1(\fA_n)$ which is also invariant for $W'$, and $W'|_\H = \ol{\beta} U_\phi$.
Thus it is clear that $(\sigma_1,W')$ is a unitary dilation of $(\lambda,  \ol{\beta} U_\phi)$.
Thus $(\sigma_1, \beta W')$ is a unitary dilation of $(\lambda, U_\phi)$.
\end{proof}

\begin{rem}
Arveson \cite{Arv1} defines a boundary representation of an operator algebra $\A$
to be an irreducible $*$-representation $\pi$ of $\ca(\A)$ so that $\pi|_\A$ has a
unique completely positive extension to $\ca(\A)$.  These are just the maximal
representations of $\A$ which are irreducible \cite{Arv3}.
So it is of interest to know when we can obtain irreducible dilations.
In Theorem~\ref{T:Cuntz dilate 1}, the  representation $\lambda_\bx$ is already
irreducible provided that $\bx$ is not eventually periodic, and
the representation $\sigma_1$ is also irreducible \cite{DP1}.
So we obtain boundary representations.
\end{rem}

 An immediate consequence of these dilation theorems,
 Theorem~\ref{main1} together with
Theorems~\ref{T:Cuntz dilate 1} and \ref{T:Cuntz dilate 2},
are the following crucial facts.

\begin{cor} \label{C:Cuntz diln}
Let $\phi \in \Aut(\fA_n)$.  Then every row contractive covariant representation
has a unitary dilation of Cuntz type.  Conversely, every covariant pair $(\sigma,U)$,
where $\sigma$ is a $*$-extendible representation of $\fA_n$ such that
$\sigma(L_1), \dots, \sigma(L_n)$ generate a copy of $\O_n$ and $U$ is
a unitary satisfying the covariance relations $\sigma(A) U = U \sigma(\phi(A))$
for all $A \in \fA_n$ determines a maximal representation of $\fA_n \times_\phi \bZ^+$.
\end{cor}

\begin{cor} \label{theenvelope}
 $\cenv(\fA_n \times_\phi \bZ^+) = \O_n\times_\phi \bZ$.
\end{cor}

\section{Concrete representations for $\fA_n \times_{\phi} \bZ^+$.}

One of the motivations for the present paper was to provide concrete
faithful representations for $\fA_n \times_{\phi} \bZ^+$.
Corollary \ref{theenvelope} essentially reduces this to the (selfadjoint) problem
of finding faithful representations for $\O_n\times_\phi \bZ$. We know one construction
of a representation of $\O_n\times_\phi \bZ$.
Just take the canonical map onto $\ca(\E_n, U_\phi)/\fK$.
When $\O_n\times_\phi \bZ$ is simple, this is an isomorphism.
We show that this is the case when $\phi$ is aperiodic.

\begin{thm}
The only unitaries in $\O_n$ which conjugate $\fA_n$ into itself are scalars.
\end{thm}

\begin{proof}
Suppose that $U$ is a unitary in $\O_n$ such that $U\fA_n U^* = \fA_n$.

Consider the atomic representation $\sigma_i$ on $\H = \bC \zeta \oplus \Fock^{(n-1)}$,
where the $\Fock^{(n-1)} = \bigoplus \{ \H_k : 1 \le k \ne i \le n \}$ and $\H_k$ has standard
basis $\xi^k_w$ for $w \in \Fn$.  We define
\[ \sigma_i(\fs_j) \zeta = \begin{cases} \zeta &\qif j=i\\ \xi^j_\mt &\qif j \ne i \end{cases} \]
and
\[ \sigma_i(\fs_j) \xi^k_w = \xi^k_{jw} \qfor k \ne i,\ w \in \Fn .\]
The significance of this representation is that $\bC \zeta$ is the unique minimal invariant
subspace for $\sigma_i(\fA_n^*)$.  Hence it must be fixed by $\sigma_i(U)$.
It follows that $\sigma_i(U\fs_iU^*) \zeta = \zeta$.  But it is immediately apparent that
the only elements of $\fA_n$ which take $\zeta$ to itself are of the form $h_i(\sigma_i(\fs_i))$
where $h_i\in \AD$ and $h_i(1)=1$.  Thus $U\fs_iU^* = h_i(\fs_i)$.

Likewise there are representations $\sigma_{ij}$ with a unique minimal minimal invariant
subspace for $\sigma_i(\fA_n^*)$ which is one dimensional $\bC \zeta$ satisfying
$\sigma_{ij}((\fs_i+\fs_j)/\sqrt2)\zeta=\zeta$.  The same argument shows that there is
an $h_{ij}\in\AD$ so that
$U \big((\fs_i+\fs_j)/ \sqrt2 \big)U^* = h_{ij}\big( (\fs_i+\fs_j)/\sqrt2 \big)$.
Therefore
\[ h_i(\fs_i) + h_j(\fs_j) = \sqrt2 h_{ij} \big((\fs_i+\fs_j)/\sqrt2 \big) .\]
It is easy to see from this that $h_i=h_j=h_{ij} = \lambda z$.
Since $h_i(1)=1$, we see that $\lambda = 1$.
Therefore $U$ lies in the centre of $\O_n$; whence $U$ is scalar.
\end{proof}

\begin{cor}
The non-trivial Voiculescu automorphisms of $\O_n$ are outer.
\end{cor}

Kishimoto \cite[Theorem~3.1]{Kish} showed that if $\fA$ is a simple C*-algebra
and $\alpha \in \Aut(\fA)$ such that $\alpha$ is \textit{aperiodic}, i.e.,
$\alpha^n$ is outer for all $n \ne 0$, then
$\fA \times_\alpha \bZ$ is simple.  Thus we obtain:

\begin{cor}   \label{final1}
If $\phi \in \Aut(\fA_n)$ is aperiodic, then $\O_n \times_\phi \bZ$ is simple,
and thus is isomorphic to $\ca(\E_n, U_\phi)/\fK$. Therefore the representation
$\pi_{\O_n}\times \hat{U}_\phi$ of Example \ref{basiceg} is a faithful representation of
$\fA_n \times_\phi \bZ$.
\end{cor}

We will now observe that the other representation of Example \ref{basiceg}, i.e., $\id \times U_\phi$
is also faithful for $\fA_n \times_\phi \bZ$, provided that $\phi$ is aperiodic. This is of course a  feature of the
non-selfadjoint theory, since $\id \times U_\phi$ comes from a representation of
$\E_n \times_{\phi} \bZ$.

\begin{cor}  \label{final2}
If $\phi \in \Aut(\fA_n)$ is aperiodic, then the representation
$\id \times U_\phi$ of Example \ref{basiceg} is a faithful representation of
$\fA_n \times_\phi \bZ$.
\end{cor}

\begin{proof}
Consider the diagram
\[
\begin{CD}
\fA_n\times_{\phi}\bZ^+@> \id \times U_\phi>> \B(\Fock) @> q >> \B(\Fock)\slash \K(\Fock),
\end{CD}
\]
where $q$ denotes the Calkin map.  By Corollary \ref{final1}, the composition
$q\circ  (\id \times U_\phi)=\pi_{\O_n}\times \hat{U}_\phi$
is isometric, and therefore $\id \times U_\phi$ is isometric as well.
\end{proof}

When $\phi$ is periodic, it may be necessary to use a family of representations.
A natural choice are $(\id,\lambda U_\phi)$ for $\lambda \in \bT$.
Form $\H = \Fock \otimes L^2(\bT)$.
Consider $(\id^{(\infty)}, U_\phi \otimes M_z)$
where $\id^{(\infty)}(A) = A \otimes I$ and $M_z$ is multiplication by $z$ on $L^2(\bT)$.
Clearly this is a covariant representation.
Let $R_\mu$ denote the operator of rotation by $\mu \in \bT$ on $L^2(\bT)$.
Then $\ad I\otimes R_\mu$ fixes $\id^{(\infty)}(\E_n)$ and conjugates $M_z$ to $\mu M_z$.
Consequently integration with respect to $\mu$ yields a faithful expectation of
$\fB = \ca(\id^{(\infty)}(\E_n), U_\phi\otimes M_z)$ onto the copy $\id^{(\infty)}(\E_n)$ of $\E_n$.
A standard gauge invariant uniqueness argument shows that $\fB \simeq \E_n \times_\phi \bZ$.
Modding out by the ideal $\fJ$ generated by $\id^{(\infty)}(\fK)$ yields a covariant
representation of $\O_n \times_\phi \bZ$ which has a faithful expectation onto $\O_n$.
Thus this also yields a faithful representation of the crossed product.
To summarize, we have established that:

\begin{prop}
The crossed product $\O_n \times_\phi \bZ$ is isomorphic to $\fB/\fJ$, where
$\fB = \ca(\id^{(\infty)}(\E_n), U_\phi\otimes M_z)$ and
$\fJ$ is the ideal generated by $\id^{(\infty)}(\fK)$.
\end{prop}


\end{document}